\documentclass[a4paper]{amsart}
\usepackage{graphicx}
\usepackage{amssymb}
\usepackage{amsmath}
\usepackage{amsthm}
\usepackage{amscd}
\usepackage[all,2cell]{xy}

\UseAllTwocells \SilentMatrices
\newtheorem{thm}{Theorem}[section]

\newtheorem{cor}[thm]{Corollary}

\newtheorem{lem}[thm]{Lemma}
\newtheorem{exm}[thm]{Example}

\newtheorem{prop}[thm]{Proposition}

\theoremstyle{definition}
\newtheorem{defn}[thm]{Definition}
\theoremstyle{remark}
\newtheorem{rem}[thm]{\bf Remark}
\numberwithin{equation}{section}

\begin{document}
\title[A non-vanishing result on the singularity category]{A non-vanishing result on the singularity category}
\author[Chen, Li, Zhang, Zhao] {Xiao-Wu Chen, Zhi-Wei Li, Xiaojin Zhang, Zhibing Zhao$^*$}

\thanks{$^*$ The corresponding author}
\date{\today}

\thanks{xwchen$\symbol{64}$mail.ustc.edu.cn, zhiweili@jsnu.edu.cn, xjzhang@jsnu.edu.cn, zbzhao@ahu.edu.cn}

\maketitle

\dedicatory{}%
\commby{}%

\begin{abstract}
We prove that a virtually periodic object in an abelian category gives rise to a non-vanishing result on certain Hom groups in the singularity category. Consequently, for any artin algebra with  infinite global dimension, its singularity category has no silting subcategory, and the associated differential graded Leavitt algebra has a non-vanishing cohomology in each degree. We verify the Singular Presilting Conjecture for singularly-minimal algebras and ultimately-closed  algebras. We obtain a trichotomy on the Hom-finiteness of the cohomology of  differential graded Leavitt algebras.
\end{abstract}

\section{Introduction}

The singularity category is a fundamental homological invariant for a ring with infinite global dimension. It is traced  back to \cite{Buc} and is rediscovered in the geometric context in \cite{Orl}. It has received increasing attention from people in different subjects.

Recall that a module is periodic if its higher syzygy is isomorphic to itself. These modules play a particular role in the singularity category \cite{Usui}. We propose a slightly more general notion: a module is called \emph{virtually periodic} if its higher syzygy lies in the extension closure of the module itself; see Definition~\ref{defn:vp}. A prototype is the semisimple quotient module of a left artinian ring modulo its Jacobson radical.

The central result of this paper is Theorem~\ref{thm}, which states the following non-vanishing property: for a virtually $d$-periodic module $M$, the Hom groups between $M$ and its $(nd)$-th suspension $\Sigma^{nd}(M)$ in the singularity category are always non-vanishing for all integers $n$.

There are two consequences of the above non-vanishing result. The first one states that the singularity category of a left artinian ring with infinite global dimension does not have a silting subcategory in the sense of \cite{KV, AI}; see Corollary~\ref{cor:silt}. This strengthens \cite[Theorem~1]{AHMW}, and partially supports the  Singular Presilting Conjecture \cite{CHQW} and thus the well-known Auslander-Reiten Conjecture \cite{AR75}.

We prove that the  Singular Presilting Conjecture and thus the Auslander-Reiten Conjecture hold for singularly-minimal artin algebras; see Proposition~\ref{prop:sin-min}.  We prove that the Singular Presilting Conjecture holds for ultimately-closed algebras in the sense of \cite{Jans}, which include periodic algebras and syzygy-finite algebras; see Proposition~\ref{prop:SPC} and Remark~\ref{rem:vuc}.

The second consequence states that the differential graded Leavitt algebra \cite{CW} associated to an artin algebra with infinite global dimension has a non-vanishing cohomology in each degree; see Proposition~\ref{prop:Leavitt}. In Section~4, we use virtually periodic objects to characterize the Hom-finiteness of the singularity category. We obtain a trichotomy on the cohomlogies of the  differential graded Leavitt algebras associated to  artin algebras; see Proposition~\ref{prop:Leavitt2}. We mention that the results on differential graded Leavitt algebras are analogous to the ones in \cite{AV} on the stable cohomology algebras of the residue fields of commutative  noetherian local rings.

We will denote the suspension functor in any triangulated category by $\Sigma$, and write dg for `differential graded'. For triangulated categories, we refer to \cite{Hap}; for artin algebras, we refer to \cite{ARS}.

\section{Virtually periodic objects}

Let $\mathcal{A}$ be an abelian category  with enough projective objects. The latter condition means that for each object $M$, there is an epimorphism $P\rightarrow M$ with $P$ projective. We denote by $\mathcal{P}$ the full subcategory formed by all projective objects, and by $\underline{\mathcal{A}}$ the stable category of $\mathcal{A}$ modulo morphisms factoring through projective objects.

For any object $M$, the first syszygy $\Omega(M)$ of $M$ is defined to be the kernel of any epimorphism $P\rightarrow M$ with $P$ projective.  We mention that $\Omega(M)$ is  uniquely defined in the stable category $\underline{\mathcal{A}}$. Denote by $\langle M\rangle$ the smallest full subcategory of $\mathcal{A}$, which contains $\{M\}\cup \mathcal{P}$ and is closed under direct summands and extensions. The extension-closed condition means that for any short exact sequence $0\rightarrow X\rightarrow Y\rightarrow Z \rightarrow 0$ with $X, Z\in \langle M\rangle$, we have that $Y$ necessarily lies in $\langle M\rangle$.

Let $d\geq 1$. Recall that a non-projective object $M$ is called \emph{$d$-periodic} if there is an isomorphism $\Omega^d(M)\simeq M$ in $\underline{\mathcal{A}}$. A $d$-periodic object necessarily has infinite projective dimension. The study of periodic modules is traced back to \cite{APR}. We mention that periodic modules  play a  role in the singularity category \cite{Usui}.

\begin{defn}\label{defn:vp}
Let $d\geq 1$. An object $M$ in $\mathcal{A}$ is said to be \emph{virtually $d$-periodic} provided that $M$ has infinite projective dimension and that $\Omega^d(M)$ lies in $\langle M\rangle$.
\end{defn}

We observe that a $d$-periodic object is virtually $d$-periodic. The following fact is easy.

\begin{lem}\label{lem:n}
Assume that $M$ is virtually $d$-periodic. Then $M$ is virtually $(nd)$-periodic for any $n\geq 1$.
\end{lem}

\begin{proof}
We use induction on $n$ and assume that $M$ is virtually $(nd)$-periodic. By applying Horseshoe Lemma and the fact that $\Omega^{nd}(M)\in \langle M\rangle$, we infer that
 $$\Omega^{(n+1)d}(M)=\Omega^d(\Omega^{nd}(M))\in \langle \Omega^d(M) \rangle\subseteq \langle M \rangle.$$
 Here, the latter inclusion uses the fact that $\Omega^d(M)\in \langle M\rangle$. We infer that $M$ is also virtually $(n+1)d$-periodic.
\end{proof}

For a left noetherian ring $\Lambda$,  we denote by $\Lambda\mbox{-mod}$ the abelian category of finitely generated left $\Lambda$-modules. The following  examples motivate Definition~\ref{defn:vp}.

\begin{exm}\label{exm:artin}
{\rm Let $\Lambda$ be a left artinian ring with infinite global dimension. Set $\Lambda_0=\Lambda/J$ with $J$ its Jacboson radical.  Then $\Lambda_0$ is  virtually $1$-periodic in $\Lambda\mbox{-mod}$.

Indeed, the projective dimension of $\Lambda_0$ is equal to the global dimension of $\Lambda$, and thus is infinite. Moreover, $\Omega(\Lambda_0)$ is clearly an iterated extension of simple modules. Then we infer that $\Omega(\Lambda_0)\in \langle \Lambda_0\rangle$.}
\end{exm}

For an object $X$ in $\mathcal{A}$, we denote by ${\rm add}\; X$  the smallest full additive subcategory which contains $X$ and is closed under direct summands.

\begin{exm}\label{exm:uc}
{\rm Recall from \cite[Section~3]{Jans} that an object $M$ in $\mathcal{A}$ is \emph{ultimately-closed},  if there exist $d\geq 1$ and some object $X$ in ${\rm add}\; (M\oplus \Omega(M)\oplus \cdots \oplus \Omega^{d-1}(M))$ such that $\Omega^d(M)$ and $X$ are isomorphic in $\underline{\mathcal{A}}$. In this case, the object $M\oplus \Omega(M)\oplus \cdots \oplus \Omega^{d-1}(M)$ is virtually $1$-periodic.

The abelian category $\mathcal{A}$ is called \emph{ultimately-closed}, if any object is ultimately-closed. Following \cite[Section~3]{Jans}, a left noetherian ring $\Lambda$ is \emph{ultimately-closed} if $\Lambda\mbox{-mod}$ is ultimately-closed.}
\end{exm}

  Recall that  a hypersurface ring is of the form $S/(x)$ with $S$ a regular local ring and $x$ a nonzero element. By \cite[Theorem~6.1]{Eis}, the higher syzygy of each module over a hypersurface ring is $2$-periodic and thus each module is ultimately-closed. Consequently,  a hypersurface ring is ultimately-closed. Recall that  a finite dimensional algebra over a field is  \emph{periodic} if it has a periodic bimodule resolution; see \cite{BBK} for concrete examples.  By a similar reasoning, we infer that a periodic algebra is ultimately-closed.

In what follows, we give further classes of   ultimately-closed rings.

\begin{exm}\label{exm:syzygy-finite}
{\rm For any $d\geq 1$, we denote by $\Omega^d(\mathcal{A})$ the full subcategory of $\mathcal{A}$ formed by those objects that are isomorphic to $\Omega^d(M)$ in $\underline{\mathcal{A}}$ for some object $M$.
The  abelian category $\mathcal{A}$ is \emph{syzygy-finite} provided that there exist $d\geq 1$ and an object $E$ such that $\Omega^d(\mathcal{A})\subseteq {\rm add}\; E$. In this case, the object $E$ is virtually $d$-periodic provided that $\mathcal{A}$ has infinite global dimension. Moreover, if such an object $E$ already belongs to $\Omega^d(\mathcal{A})$, then $E$ is virtually $1$-periodic.

We observe that a syzygy-finite Krull-Schmidt abelian category is necessarily ultimately-closed; compare \cite[p.73]{AR75}. For example, let $\Lambda$ be  a left artinian ring which is \emph{syzygy-finite}, that is, $\Lambda\mbox{-mod}$ is syzygy-finite. Then  the module category $\Lambda\mbox{-mod}$ is ultimately-closed, and thus $\Lambda$ is ultimately-closed.

We mention that syzygy-finite artinian rings include artinian rings of finite representation type,  artinian  rings with square-zero Jacobson radical, and finite dimensional monomial algebras  by \cite[Theorem~I]{Zim}.}
\end{exm}

\begin{exm}\label{exm:CM}
{\rm We assume that $R$ is  a commutative noetherian complete local ring,  which is non-regular and  Cohen-Macaulay of finite Cohen-Macaulay-type \cite{AR}.  We  take $E$  to be the direct sum of all indecomposable maximal Cohen-Macaulay $R$-modules, and let $d$  be the Krull dimension of $R$.  Then we have $\Omega^d(R\mbox{-mod})\subseteq {\rm add}\; E$, because the $d$-th syzygy of any finitely generated $R$-module is maximal Cohen-Macaulay. In particular, the category $R\mbox{-mod}$ is syzygy-finite. The completeness of $R$ implies that $R\mbox{-mod}$ is Krull-Schmidt, and thus is ultimately-closed. Therefore, the ring $R$ is ultimately-closed.

 We assume in addition that $R$ is Gorenstein. Then we have $\Omega^d(R\mbox{-mod})={\rm add}\; E$, in which case $E$ is virtually $1$-periodic. For another choice of such a module $E$, we refer to \cite[Proposition~3.2]{Kal}.}
\end{exm}

Denote by $\mathbf{D}^b(\mathcal{A})$ the bounded derived category of $\mathcal{A}$. Using the canonical functor, we  view the bounded homotopy category $\mathbf{K}^b(\mathcal{P})$ as a thick triangulated subcategory of $\mathbf{D}^b(\mathcal{A})$. Following \cite{Buc}, the \emph{singularity category} of $\mathcal{A}$ is defined to be the following Verdier quotient triangulated category
$$\mathbf{D}_{\rm sg}(\mathcal{A})=\mathbf{D}^b(\mathcal{A})/\mathbf{K}^b(\mathcal{P}).$$
As usual, we identify any object $M$ in $\mathcal{A}$ with the corresponding stalk complex concentrated in degree zero, which is still denoted by $M$. The latter is also viewed as an object in $\mathbf{D}_{\rm sg}(\mathcal{A})$. Consequently,  $\Sigma^n(M)$ will mean the corresponding stalk complex concentrated in degree $-n$.

The following fact is well known; compare \cite[Lemma~2.2]{Chen} and \cite[Lemma~2.2.2]{Buc}.

\begin{lem}\label{lem:sing}
Let $M$ be an object in $\mathcal{A}$. Then there is an isomorphism $M\simeq \Sigma\Omega(M)$ in $\mathbf{D}_{\rm sg}(\mathcal{A})$.
\end{lem}

\begin{proof}
Recall that any short exact sequence in $\mathcal{A}$ induces canonically an exact triangle in $\mathbf{D}^b(\mathcal{A})$ and thus an exact triangle in $\mathbf{D}_{\rm sg}(\mathcal{A})$. We consider the following exact sequence $0\rightarrow \Omega(M)\rightarrow P\rightarrow M\rightarrow 0$. As the projective object $P$ vanishes in $\mathbf{D}_{\rm sg}(\mathcal{A})$, the induced exact triangle in $\mathbf{D}_{\rm sg}(\mathcal{A})$ is the form
$$\Omega(M) \longrightarrow 0\longrightarrow M\longrightarrow \Sigma\Omega(M).$$
By \cite[Lemma~I.1.7]{Hap}, we infer that the morphism $M\rightarrow \Sigma\Omega(M)$ is an isomorphism, as required. \end{proof}

\begin{thm}\label{thm}
Let $M$ be a virtually $d$-periodic object in $\mathcal{A}$. Then we have
$${\rm Hom}_{\mathbf{D}_{\rm sg}(\mathcal{A})}(M, \Sigma^{nd}(M))\neq 0$$
for any integer $n$.
\end{thm}

\begin{proof}
In view of Lemma~\ref{lem:n}, it suffices to prove that ${\rm Hom}_{\mathbf{D}_{\rm sg}(\mathcal{A})}(M, \Sigma^d(M))\neq 0\neq {\rm Hom}_{\mathbf{D}_{\rm sg}(\mathcal{A})}(M, \Sigma^{-d}(M))$.

Since $M$ has infinite projective dimension, it does not vanish in $\mathbf{D}_{\rm sg}(\mathcal{A})$. Consequently, we have ${\rm Hom}_{\mathbf{D}_{\rm sg}(\mathcal{A})}(M, M)\neq 0$. By Lemma~\ref{lem:sing}, we have an isomorphism  $M\simeq \Sigma^d\Omega^d(M)$ in $\mathbf{D}_{\rm sg}(\mathcal{A})$. In particular, we have
\begin{align}\label{equ:M}
{\rm Hom}_{\mathbf{D}_{\rm sg}(\mathcal{A})}(M, \Sigma^d\Omega^d(M))\neq 0.
\end{align}

We claim that ${\rm Hom}_{\mathbf{D}_{\rm sg}(\mathcal{A})}(M, \Sigma^d(M))\neq 0$. Otherwise, the object $M$ belongs to the following full subcategory
$$\mathcal{S}=\{X\in \mathcal{A}\; |\;{\rm Hom}_{\mathbf{D}_{\rm sg}(\mathcal{A})}(M, \Sigma^d(X))=0\}.$$
As any short exact sequence in $\mathcal{A}$ induces an exact triangle in $\mathbf{D}_{\rm sg}(\mathcal{A})$, it follows that $\mathcal{S}$ is closed under extensions. Clearly, it contains $\mathcal{P}$ and is closed under direct summands. It follows that $\mathcal{S}$ contains $\langle M\rangle$. Since $M$ is virtually $d$-periodic, we have that $\Omega^d(M)$ belongs to $\langle M\rangle$ and thus is contained in $\mathcal{S}$. This contradicts to the inequality (\ref{equ:M}), and proves the claim.

 Dually, we observe that
 $$
{\rm Hom}_{\mathbf{D}_{\rm sg}(\mathcal{A})}( \Sigma^d\Omega^d(M), M)\neq 0.
$$
Consider the full subcategory
$$\mathcal{S}'=\{Y\in \mathcal{A}\; |\;{\rm Hom}_{\mathbf{D}_{\rm sg}(\mathcal{A})}(\Sigma^d(Y), M)=0\},$$
which contains $\mathcal{P}$ and is closed under direct summands and extensions. By a dual argument as above, we prove that
$${\rm Hom}_{\mathbf{D}_{\rm sg}(\mathcal{A})}(\Sigma^d(M), M)\neq 0.$$
This completes the whole proof.
\end{proof}

\section{Two consequences}

In this section, we draw two consequences of Theorem~\ref{thm}. We show that the singularity category of a left artinian ring with infinite global dimension does not have a silting subcategory;  see Corollary~\ref{cor:silt}. We verify the Singular Presilting Conjecture for two classes of artin algebras: singularly-minimal algebras and ultimately-closed  algebras; see Propositions~\ref{prop:sin-min} and~\ref{prop:SPC}. We prove that the dg Leavitt algebra \cite{CW} associated to any artin algebra with infinite global dimension has a non-vanishing cohomology in each degree; see Proposition~\ref{prop:Leavitt}.

\subsection{Silting subcategories}

Let $\mathcal{T}$ be a triangulated category. Recall from \cite[Definition~2.1]{AI} that a full additive subcategory $\mathcal{M}$ is called \emph{silting}, provided that the following two conditions are fulfilled.
\begin{enumerate}
\item The subcategory $\mathcal{M}$ is \emph{presilting}, that is, ${\rm Hom}_\mathcal{T}(M, \Sigma^n(M))=0$ for any $M\in \mathcal{M}$ and $n>0$.
\item The subcategory $\mathcal{M}$ \emph{generates} $\mathcal{T}$ in the sense that $\mathcal{T}$ itself is the smallest thick triangulated subcategory containing $\mathcal{M}$.
\end{enumerate}
An object $X$ is  called presilting (respectively, silting) provided that ${\rm add}\; X$ is a presilting (respectively, silting) subcategory. We mention that the study of silting objects goes back to \cite{KV}.

The following result is due to \cite[Proposition~2.4]{AI}.

\begin{lem}\label{lem:AI}
Assume that $\mathcal{T}$ has a silting subcategory. Then for any object $X$, ${\rm Hom}_\mathcal{T}(X, \Sigma^d(X))=0$ for sufficiently large $d$. \hfill $\square$
\end{lem}

In what follows, $\mathcal{A}$ is an abelian category with enough projective objects. We have the first consequence of Theorem~\ref{thm}.

\begin{prop}\label{prop:silt}
Assume that $\mathcal{A}$ contains a virtually $d$-periodic object for some $d\geq 1$. Then $\mathbf{D}_{\rm sg}(\mathcal{A})$ has no silting subcategory.
\end{prop}

\begin{proof}
Take $M$ be a virtually $d$-periodic object in $\mathcal{A}$.  Theorem~\ref{thm} implies that
$${\rm Hom}_{\mathbf{D}_{\rm sg}(\mathcal{A})}(M, \Sigma^{nd}(M))\neq 0$$
for any integer $n$. In particular, $n$ can be sufficiently large. In view of Lemma~\ref{lem:AI}, we have the required non-existence of a silting subcategory.
\end{proof}

For a left noetherian ring $\Lambda$, we usually write $\mathbf{D}_{\rm sg}(\Lambda)$ for $\mathbf{D}_{\rm sg}(\Lambda\mbox{-mod})$.

The following result strengthens \cite[Theorem~1]{AHMW}, where the corresponding result is proved under a finiteness assumption on the selfinjective dimension; compare \cite[Example~2.5(b)]{AI} and \cite[Corollary~3.12]{IY}. The argument here is completely different.

\begin{cor}\label{cor:silt}
Let $\Lambda$ be a left artinian ring with infinite global dimension. Then $\mathbf{D}_{\rm sg}(\Lambda)$ has no silting subcategory.
\end{cor}

\begin{proof}
By Example~\ref{exm:artin}, the semisimple $\Lambda$-module $\Lambda_0$ is virtually $1$-periodic. Then we apply Proposition~\ref{prop:silt}.
\end{proof}

The above non-existence partially supports the following conjecture \cite{CHQW}; compare \cite{IY}.

\vskip 5pt

 \emph{Singular Presilting Conjecture.} For any artin algebra $\Lambda$, there is no nonzero presilting subcategory in $\mathbf{D}_{\rm sg}(\Lambda)$.

\vskip 5pt

By \cite[Lemma~3.4]{IY} or  \cite[Proposition~1.21]{Orl},  this conjecture implies the following well-known conjecture, proposed in \cite[p.70]{AR75};  compare \cite[Section~1]{CHQW}.
 \vskip 5pt

 \emph{Auslander-Reiten Conjecture.} For a non-projective module $M$ over any artin algebra $\Lambda$, we have ${\rm Ext}^n_\Lambda(M, M\oplus \Lambda)\neq 0$ for some $n\geq 1$.

\vskip 5pt
We mention that, by \cite[Theorem~4.4]{Buc}, the Singular Presilting Conjecture for $\Lambda$ is equivalent to  the Auslander-Reiten Conjecture for $\Lambda$, provided that  $\Lambda$ is a Gorenstein artin algebra.

We say that a left artinian ring $\Lambda$ is \emph{singularly-minimal}, if any thick subcategory of $\mathbf{D}_{\rm sg}(\Lambda)$ is equal to either zero or $\mathbf{D}_{\rm sg}(\Lambda)$ itself.  For example, by \cite[Example~3.11]{Chen} the algebra $k[x_1, \cdots, x_n]/{(x_ix_j, 1\leq i, j\leq n)}$ over a field $k$ is singularly-minimal. For more such examples of trivial extension algebras, we refer to \cite[Corollary~4.3]{Chen16}.

The following result shows that the Singular Presilting Conjecture and thus the Auslander-Reiten Conjecture hold for singularly-minimal artin algebras.

\begin{prop}\label{prop:sin-min}
Let $\Lambda$ be a singularly-minimal left artianian ring. Then there is no nonzero presilting subcategory in $\mathbf{D}_{\rm sg}(\Lambda)$.
\end{prop}

\begin{proof}
By the singularly-minimality of $\Lambda$, any nonzero presilting subcategory in $\mathbf{D}_{\rm sg}(\Lambda)$ is silting. Then we apply Corollary~\ref{cor:silt}.
\end{proof}

The following result implies that the Singular Presilting Conjecture holds for ultimately-closed artin algebras; compare \cite[Proposition~1.3]{AR75}. Recall that finite dimensional periodic algebras and syzygy-finite artin algebras are ultimately-closed; see Section~2 or \cite[p.73]{AR75}.

\begin{prop}\label{prop:SPC}
Assume that $\mathcal{A}$ is ultimately-closed. Then $\mathbf{D}_{\rm sg}(\mathcal{A})$ has no nonzero presilting subcategory. In particular, for a ultimately-closed ring $\Lambda$, there is no nonzero presilting subcategory in $\mathbf{D}_{\rm sg}(\Lambda)$.
\end{prop}

\begin{proof}
It suffices to prove that  $\mathbf{D}_{\rm sg}(\mathcal{A})$ has no nonzero presilting object. Assume that $X$ is a nonzero presilting object in $\mathbf{D}_{\rm sg}(\mathcal{A})$. Set $\mathcal{T}$ to be the smallest thick subcategory containing $X$.  By \cite[Lemma~2.1]{Chen}, there exists an object $M$ in $\mathcal{A}$ such that $X$ is isomorphic to $\Sigma^n(M)$ for some integer $n$.  Since $M$ is ultimately-closed, there exists $d\geq 1$ such that $E=M\oplus \Omega(M)\oplus \cdots\oplus \Omega^{d-1}(M)$ is virtually $1$-periodic; see Example~\ref{exm:uc}. In view of Lemma~\ref{lem:sing}, the object $E$ belongs to $\mathcal{T}$. However, by Theorem~\ref{thm}, we have
$${\rm Hom}_{\mathcal{T}}(E, \Sigma^{n}(E))\neq 0$$
for any integer $n$. Since $X$ is a silting object of $\mathcal{T}$, we have a desired contradiction by Lemma~\ref{lem:AI}.
\end{proof}

\begin{rem}\label{rem:vuc}
It seems that Proposition~\ref{prop:SPC} might be strengthened. The abelian category $\mathcal{A}$ is said be  \emph{virtually ultimately-closed} if for each object $M$, there exists $d\geq 1$ such that $\Omega^d(M)$ lies in $\langle M\oplus \Omega(M)\oplus \cdots\oplus \Omega^{d-1}(M)\rangle$. In this case, the object $M\oplus \Omega(M)\oplus \cdots\oplus \Omega^{d-1}(M)$ is still virtually $1$-periodic. Then the same argument above implies that Proposition~\ref{prop:SPC} holds for virtually ultimately-closed categories. However, we do not know any virtually ultimately-closed category, which is not ultimately-closed.
\end{rem}

\begin{rem}
 Unlike the ungraded case, the graded singularity category of a graded artin algebra might have a silting object;  for example, see \cite[Theorem~3.0.3]{LZ} and \cite[Theorem~1.4]{KMY}.
\end{rem}

We mention the following known non-existence result.

\begin{prop}
Let $R$ be a commutative noetherian  local ring, which is non-regular. Then  $\mathbf{D}_{\rm sg}(R)$ has no silting subcategory.
\end{prop}

\begin{proof}
Denote by $k$ the residue field of $R$. By \cite[Theorem~6.5]{AV}, we infer that
\begin{align}\label{equ:comm.ring}
{\rm Hom}_{\mathbf{D}_{\rm sg}(R)}(k, \Sigma^n(k))\neq 0
 \end{align}
 for any integer $n$. Here, we identify ${\rm Hom}_{\mathbf{D}_{\rm sg}(R)}(k, \Sigma^n(k))$ with the stable cohomology group $\widehat{\rm Ext}^n_R(k, k)$; see \cite[1.4.2]{AV}. Then the non-existence follows from Lemma~\ref{lem:AI}.
\end{proof}

\subsection{The dg Leavitt algebra}\label{subsec:3.2}

In this subsection, we assume that  $\Lambda$ is an artin algebra over a commutative artinian ring $k$.

 Recall that $\Lambda_0 = \Lambda/J$ with $J$ its Jacobson radical.  We will assume that $\Lambda_0$ is a subalgebra of $\Lambda$ with a decomposition $\Lambda=\Lambda_0\oplus J$ of  $\Lambda_0$-$\Lambda_0$-bimodules.  This assumption holds if $\Lambda$ is given by a finite quiver with admissible relations, or if $\Lambda$ is a finite dimensional algebra  over a perfect field.

Consider the left $\Lambda_0$-dual $J^*={\rm Hom}(J, \Lambda_0)$ of $J$, which carries a natural $\Lambda_0$-$\Lambda_0$-bimodule structure. We have the \emph{Casimir element} $c=\sum_{i\in S} \alpha_i^*\otimes \alpha_i\in J^*\otimes J$, where $\{\alpha_i\; |\; i\in S\}$ and $\{\alpha_i^*\; |\; i\in S\}$ form the dual basis of $J$.  The multiplication on $J$ induces a map of $\Lambda_0$-$\Lambda_0$-bimodules
$$\partial_+\colon J^*\longrightarrow J^*\otimes  J^*.$$
To be more precise, we have  $\partial_+(g)=\sum g_1\otimes g_2$ such that $g(ab)=\sum g_2(ag_1(b))$ for any $a, b\in J$.

Associated to the artin algebra $\Lambda$,  the \emph{dg Leavitt algebra} $L=L_{\Lambda_0}(J)$ is introduced in \cite{CW}. As an algebra, it is given by
$$L=T_{\Lambda_0}(J\oplus J^*)/(a\otimes g-g(a),\; 1-c \; |\; a\in J, g\in J^*).$$
Here, $T_{\Lambda_0}(J\oplus J^*)$ denotes the tensor algebra. It is naturally $\mathbb{Z}$-graded such that $|e|=0$ for any $e\in \Lambda_0$, $|a|=-1$ for any $a\in J$ and $|g|=1$ for any $g\in  J^*$. The  differential $\partial$ on $L$ is uniquely determined by the graded Leibniz rule and the conditions that $\partial|_{\Lambda_0}=0$ and $\partial|_{J^*}=\partial_+$; see \cite[Remark~3.6]{CW}. We mention that the classical Leavitt algebras appear already in \cite{Lea}.

For any dg algebra $A$, we denote by $H^*(A)=\oplus_{n\in\mathbb{Z}} H^n(A)$ its total cohomology, which inherits a graded algebra structure from $A$. Recall that $A$ is \emph{acyclic} if $H^*(A)=0$, which is equivalent to the condition that $H^0(A)=0$.

We view $\Lambda_0$ as the corresponding stalk complex concentrated in degree zero, and as an object in $\mathbf{D}_{\rm sg}(\Lambda)$. Then the following graded $k$-module
\begin{align}\label{equ:Hom}
\bigoplus_{n\in \mathbb{Z}} {\rm Hom}_{\mathbf{D}_{\rm sg}(\Lambda)} (\Lambda_0, \Sigma^n(\Lambda_0))
\end{align}
becomes a graded $k$-algebra, whose multiplication is induced by composition of morphisms in $\mathbf{D}_{\rm sg}(\Lambda)$.

The following result is implicitly contained in \cite{CW}, and indicates the intimate link between  dg Leavitt algebras and singularity categories.

\begin{lem}\label{lem:L}
Keep the notation as above. Then there is an isomorphism of graded algebras
$$H^*(L)^{\rm op}\simeq \bigoplus_{n\in \mathbb{Z}} {\rm Hom}_{\mathbf{D}_{\rm sg}(\Lambda)} (\Lambda_0, \Sigma^n(\Lambda_0)),$$
where $H^*(L)^{\rm op}$ denotes the opposite algebra of $H^*(L)$.
\end{lem}

\begin{proof}
By the isomorphism in \cite[Theorem~9.5]{CW}, we infer that $H^*(L)^{\rm op}$ is isomorphic to the total cohomology algebra of the dg endomorphism algebra of $\Lambda_0$ in the singular Yoneda dg category. By the quasi-equivalence in \cite[Corollary~9.3]{CW}, we infer that the latter algebra is isomorphic to (\ref{equ:Hom}).
\end{proof}

In general, the structure of the dg Leavitt algebra $L$ seems to be very complicated.  The following second consequence of Theorem~\ref{thm} is  a dichotomy on its total cohomology $H^*(L)$. We refer to Proposition~\ref{prop:Leavitt2} below for a strengthened version.

\begin{prop}\label{prop:Leavitt}
Let $\Lambda$ be an artin algebra with $L$ the associated dg Leavitt algebra. Then the following statements hold.
\begin{enumerate}
\item The dg Leavitt algebra $L$ is acyclic if and only if $\Lambda$ has finite global dimension.
\item If $\Lambda$ has infinite global dimension, then $H^n(L)\neq 0$ for any integer $n$.
\end{enumerate}
\end{prop}

\begin{proof}
In view of Lemma~\ref{lem:L}, the dg Leavitt algebra $L$ is acyclic if and only if ${\rm Hom}_{\mathbf{D}_{\rm sg}(\Lambda)} (\Lambda_0, \Lambda_0)=0$, which is equivalent to the vanishing of $\Lambda_0$ in $\mathbf{D}_{\rm sg}(\Lambda)$. Since $\Lambda_0$ generates $\mathbf{D}_{\rm sg}(\Lambda)$, the last condition is equivalent to the vanishing of $\mathbf{D}_{\rm sg}(\Lambda)$, which is well known to be further equivalent to the finiteness of the global dimension of $\Lambda$. In summary, we infer (1).

For (2), we assume that $\Lambda$ has infinite global dimension.  By Example~\ref{exm:artin}, $\Lambda_0$ is virtually $1$-periodic. Theorem~\ref{thm} implies that
\begin{align}\label{equ:artin}
{\rm Hom}_{\mathbf{D}_{\rm sg}(\Lambda)} (\Lambda_0, \Sigma^n(\Lambda_0))\neq 0
\end{align}
for any integer $n$. Now the required statement follows from the isomorphism in Lemma~\ref{lem:L} immediately.
\end{proof}

\begin{rem}\label{rem:AV}
The inequality (\ref{equ:artin}) is analogous to (\ref{equ:comm.ring}). In the same spirit, Proposition~\ref{prop:Leavitt} is analogous to the characterization of regular local rings in \cite[Theorem~6.5]{AV} via the stable cohomology algebras of the residue fields.
\end{rem}

\section{The Hom-finiteness}

Let $k$ be a commutative ring. We will assume that the abelian category $\mathcal{A}$ is $k$-linear. Consequently, the singularity category $\mathbf{D}_{\rm sg}(\mathcal{A})$ is $k$-linear.  We study the Hom-finiteness of the singularity category,  and obtain a trichotomy on the cohomologies of the dg Leavitt algebras; see Proposition~\ref{prop:Leavitt2}.

For an object $M$ in $\mathcal{A}$ and $d\geq 1$, we consider a graded $k$-algebra
\begin{align}\label{equ:gamma}
\Gamma(M;d)=\bigoplus_{n\in \mathbb{Z}} {\rm Hom}_{\mathbf{D}_{\rm sg}(\mathcal{A})}(M, \Sigma^{nd}(M)),
\end{align}
whose multiplication is induced by composition of morphisms in $\mathbf{D}_{\rm sg}(\mathcal{A})$.

The proof of the following lemma resembles the one of Theorem~\ref{thm}.

\begin{lem}\label{lem:Hom-inf}
Let $d\geq 1$ and  $M$ be a virtually $d$-periodic object in $\mathcal{A}$. Assume that $X$ and $Y$ are objects in $\mathbf{D}_{\rm sg}(\mathcal{A})$. Then the following statements hold.
\begin{enumerate}
\item Assume that the $k$-module ${\rm Hom}_{\mathbf{D}_{\rm sg}(\mathcal{A})}(X, M)$ is of infinite length. Then so is ${\rm Hom}_{\mathbf{D}_{\rm sg}(\mathcal{A})}(X, \Sigma^{nd}(M))$ for each $n\geq 0$.
\item  Assume that the $k$-module ${\rm Hom}_{\mathbf{D}_{\rm sg}(\mathcal{A})}(M, Y)$ is of infinite length. Then so is ${\rm Hom}_{\mathbf{D}_{\rm sg}(\mathcal{A})}(\Sigma^{nd}(M), Y)$ for each $n\geq 0$.
\item Each homogeneous component of $\Gamma(M; d)$ is of infinite length if and only if so is one of the homogeneous components.
\end{enumerate}
\end{lem}

\begin{proof}
We will only give the proof of (1), as (2) is proved dually and (3) follows immediately by combining (1) and (2).

We now prove (1). Thanks to Lemma~\ref{lem:n},  it suffices to claim that the $k$-module ${\rm Hom}_{\mathbf{D}_{\rm sg}(\mathcal{A})}(X, \Sigma^{d}(M))$ is of infinite length. By Lemma~\ref{lem:sing}, $M$ is isomorphic to $\Sigma^d\Omega^d(M)$. Therefore, $\Omega^d(M)$ does not belong to the following full subcategory
$$\mathcal{S}''=\{Z\in \mathcal{A}\; |\; {\rm Hom}_{\mathbf{D}_{\rm sg}(\mathcal{A})}(X, \Sigma^d(Z)) \mbox{ is of finite length}\}.$$
We observe that $\mathcal{S}''$ contains $\mathcal{P}$ and is closed under direct summands and extensions. Since $\Omega^d(M)\in \langle M\rangle$ and $\Omega^d(M)$ does not belong to $\mathcal{S}''$, it follows that $M$ does not belong to $\mathcal{S}''$ either. This proves the claim and thus (1).
\end{proof}

We say that $\mathbf{D}_{\rm sg}(\mathcal{A})$ is \emph{Hom-finite} over $k$, if the $k$-module ${\rm Hom}_{\mathbf{D}_{\rm sg}(\mathcal{A})}(X, Y)$ is of finite length for any object $X$ and $Y$.

 The following result characterizes the Hom-finiteness of  $\mathbf{D}_{\rm sg}(\mathcal{A})$   using  virtually periodic objects.

\begin{prop}\label{prop:Hom-fin}
Let  $M$ be a virtually $1$-periodic object in $\mathcal{A}$ which generates $\mathbf{D}_{\rm sg}(\mathcal{A})$. Then the following statements are equivalent.
\begin{enumerate}
\item The category $\mathbf{D}_{\rm sg}(\mathcal{A})$ is Hom-finite over $k$.
\item One of the homogeneous components of $\Gamma(M; 1)$ is nonzero and of finite length.
\item Each homogeneous component of $\Gamma(M; 1)$ is nonzero and of finite length.
\end{enumerate}
\end{prop}

\begin{proof}
By Theorem~\ref{thm}, each homogeneous component of $\Gamma(M; 1)$ is nonzero. Since $M$ generates $\mathbf{D}_{\rm sg}(\mathcal{A})$, it is a standard fact that $\mathbf{D}_{\rm sg}(\mathcal{A})$ is Hom-finite if and only if ${\rm Hom}_{\mathbf{D}_{\rm sg}(\mathcal{A})}(M, \Sigma^n(M))$ is of finite length for any integer $n$. This proves ``$(1)\Leftrightarrow (2)$". The implications ``$(2)\Leftrightarrow (3)$" follow from Lemma~\ref{lem:Hom-inf}(3).
\end{proof}

In what follows, we assume that $\Lambda$ is an artin algebra over a commutative artinian ring $k$. We keep the setup in Subsection~\ref{subsec:3.2}.  We mention that the Hom-finiteness of the singularity category of certain artin algebras is studied in \cite[Section~5]{Chen}.

We have the following trichotomy on the Hom-finiteness of the  total cohomology algebra $H^*(L)$ of the associated dg Leavitt algebra.

\begin{prop}\label{prop:Leavitt2}
Let $\Lambda$ be an artin algebra and $L$ be the associated dg Leavitt algebra. Then the following statements hold.
\begin{enumerate}
\item The  algebra $\Lambda$ has finite global dimension if and only if $H^*(L)=0$.
\item The  algebra $\Lambda$ has infinite global dimension and $\mathbf{D}_{\rm sg}(\Lambda)$ is Hom-finite if and only if each homogeneous component of  $H^*(L)$ is nonzero and of finite length.
\item The category $\mathbf{D}_{\rm sg}(\Lambda)$ is not Hom-finite if and only if each homogeneous component of  $H^*(L)$ is of infinite length.
\end{enumerate}
\end{prop}

\begin{proof}
By Example~\ref{exm:artin}, $\Lambda_0$ is virtually $1$-periodic and it clearly generates $\mathbf{D}_{\rm sg}(\Lambda)$. By Lemma~\ref{lem:L}, we identify $H^*(L)^{\rm op}$ with $\Gamma(\Lambda_0; 1)$ defined in (\ref{equ:gamma}). Then the results follow from Propositions~\ref{prop:Leavitt} and \ref{prop:Hom-fin}.
\end{proof}

\begin{rem}
We point that Proposition~\ref{prop:Leavitt2}(2) is somehow analogous to the characterization of Gorenstein local rings in  \cite[Theorem~6.4]{AV}; compare Remark~\ref{rem:AV}.

Assume that $\Lambda$ is a Gorenstein artin algebra. By \cite[Theorem~4.4]{Buc}, the singularity category $\mathbf{D}_{\rm sg}(\Lambda)$ is triangle equivalent to the stable category of Gorenstein-projective $\Lambda$-modules. In particular, $\mathbf{D}_{\rm sg}(\Lambda)$ is Hom-finite and thus each homogeneous component of $H^*(L)$ is of finite length. In general, unlike the commutative case, the Hom-finiteness of $\mathbf{D}_{\rm sg}(\Lambda)$ does not imply the Gorensteinness of $\Lambda$; see \cite[Example~4.3]{Chen09} for concrete examples.
\end{rem}

\vskip 15pt

\noindent{\bf Acknowledgements}.\quad The authors thank Takuma Aihara, Hongxing Chen, Wei Hu and  Zhengfang Wang for  many helpful discussions. This work is supported by  the National Natural Science Foundation of China (No.s 12171207, 12131015 and 12161141001).

\bibliography{}

\vskip 10pt

 {\footnotesize \noindent  Xiao-Wu Chen\\
 Key Laboratory of Wu Wen-Tsun Mathematics, Chinese Academy of Sciences,\\
 School of Mathematical Sciences, University of Science and Technology of China, Hefei 230026, Anhui, PR China\\

 \footnotesize \noindent Zhi-Wei Li, Xiaojin Zhang\\
School of Mathematics and Statistics, Jiangsu Normal University, Xuzhou 221116, Jiangsu, PR China\\

 \footnotesize \noindent Zhibing Zhao\\
 School of Mathematical Sciences, Anhui University,  Hefei 230601, Anhui, PR China}

\end{document}